\documentclass[10pt]{amsart}
\usepackage[a4paper, total={6in, 8in}]{geometry}
\usepackage{amsmath}
\usepackage{amssymb}
\usepackage{amsfonts}
\usepackage{mathrsfs}
\usepackage{amscd}
\usepackage{graphicx}
\usepackage[shortlabels]{enumitem}
\usepackage{mathtools}
\usepackage{tikz-cd}
\usepackage{cleveref}
\usepackage{pstricks-add}
\usepackage{pgf,tikz}
\usepackage{subcaption}
\usepackage{bm}

\usepackage{tkz-tab}
\usepackage{xpatch}
\xpatchcmd{\tkzTabLine}{$0$}{$\bullet$}{}{}
\tikzset{t style/.style={style=solid}}

\usetikzlibrary{arrows}

\crefformat{section}{\S#2#1#3} 
\crefformat{subsection}{\S#2#1#3}
\crefformat{subsubsection}{\S#2#1#3}

\numberwithin{equation}{section}       

\theoremstyle{plain} 
\newcommand{\thistheoremname}{}
\newtheorem*{genericthm*}{\thistheoremname}
\newenvironment{namedtheorem*}[1]
  {\renewcommand{\thistheoremname}{#1}%
   \begin{genericthm*}}
  {\end{genericthm*}}

\theoremstyle{plain}
\newtheorem{theo}{Theorem}
\newtheorem{prop}{Proposition}[section]

\newtheorem{coro}[prop]{Corollary}
\newtheorem{lemm}[prop]{Lemma}

\theoremstyle{definition}

\newtheorem{defi}[prop]{Definition}
\theoremstyle{remark}

\newtheoremstyle{citing}
  {3pt}
  {3pt}
  {\itshape}
  {}
  {\bfseries}
  {.}
  {.5em}
  {\thmnote{#3}}
\theoremstyle{citing}
\newcommand{\N}{\mathbb{N}}

\newcommand{\R}{\mathbb{R}}

\newcommand{\hT}{\widehat{T}}

\newcommand{\teta}{\widetilde{\teta}}

\newcommand{\e}{\varepsilon}

\renewcommand{\=}{ : = }

\DeclareMathOperator{\dist}{dist}

\tikzset{
  declare function={
    sgn(\x) = (and(\x<0, 1) * -1) +
    (and(\x>0, 1) * 1) +
    (and(\x==0, 1) * 0);
  }
}

\usepackage{lipsum}
\begin{document}


\usetikzlibrary{shapes, arrows, calc, arrows.meta, fit, positioning, quotes} 
\tikzset{  
    state/.style ={ellipse, draw, minimum width = 0.9 cm}, 
    point/.style = {circle, draw, inner sep=0.18cm, fill, node contents={}},  
    bidirected/.style={Latex-Latex,dashed}, 
    el/.style = {inner sep=2.5pt, align=right, sloped}  
}  

\colorlet{ColorGray}{gray!10}

\title{Non-existence of wandering intervals for asymmetric unimodal maps}
\author{Jorge Olivares-Vinales} \thanks{} 
\address{Shanghai Center for Mathematical Sciences, Jiangwan Campus, Fudan University, No 2005 Songhu Road, Shanghai, China 200438 }
\email{jolivaresv@fudan.edu.cn}

\author{Weixiao Shen} \thanks{} 
\address{Shanghai Center for Mathematical Sciences, New Cornerstone Science 
Laboratory, Jiangwan Campus, Fudan University, No 2005 Songhu Road, Shanghai, China 200438 }
\email{wxshen@fudan.edu.cn}

\begin{abstract}
  We prove that an asymmetric unimodal map has no wandering 
  intervals.
\end{abstract}

\maketitle


\section{Introduction}
\label{Section_Introduction}
A natural problem in dynamical systems is the classification of ``equivalent"
systems. In the real one-dimensional case, the ordering of the phase space 
(the circle or a compact interval) gives raise to the question when
the order of the orbits of the systems determines it.
More precisely, we 
can ask whether two maps which are combinatorially equivalent are 
topologically conjugated. 

This problem goes back to Poincare's work dealing
with circle homeomorphisms, and a positive answer depends
on the non-existence of wandering intervals. The first result in this 
direction is due to Denjoy (1932) \cite{Denjoy1932_wandering-intervals} 
where he proved the non-existence of wandering intervals for a $C^1$ 
diffeomorphism $f$ of the circle  for which $\log|Df|$ has bounded variation. Later (1963),
Schwartz \cite{Schwartz_A1963_non-existence_of_wandering_intervals} gave a 
different proof of Denjoy result. His proof requires a slightly stronger smoothness 
hypothesis, but it can be extended to more general settings. The smoothness assumptions guarantee
distortion control of high iterates on certain intervals.

In the setting of maps with 
critical points, additional techniques, notably cross-ratio distortion estimates, were developed. 
Guckenheimer \cite{Guckenheimer1979_sensitive_dependence}
proved the non-existence of wandering intervals for unimodal maps with 
negative Schwarzian derivative and no critical points of inflexion type,
see also \cite{Misiurewicz1981_structure_of_mapping}. Yoccoz 
\cite{Yoccoz1984_Denjoy_counterexample} proved
the non-existence of wandering intervals for $C^{\infty}$ homeomorphisms 
of the circle having only non-flat critical points.
de Melo and van Strien \cite{deMelo-vanStrien1989_Structure_theorem} proved 
the same result for smooth unimodal maps with non-flat critical points. 
Blokh and Lyubich \cite{Lyubich1989_non-existence_of_wandering_intervals, Blokh-Lyubich_non-existence_of_wandering_intervals} proved the result in
the case of $C^2$ multimodal maps with a finite number of non-flat
critical points of turning type. Martens, de Melo, and van Strien 
\cite{Martens_de_Melo_van_Strien_Julia_Fatou_Sullivan} proved the 
non-existence of wandering intervals for smooth interval maps with 
non-flat critical points, see also
\cite{vanStrien-Vargas2004_multimodal_maps}.

In all the results for maps with critical points mentioned above, it is an
essential assumption that near
each critical point the map $f$ under consideration is almost symmetric. That is, if $c$ is a 
critical point of $f$, then   
\[ \frac{|f(c+ \e) - f(c)|}{|f(c-\e)-f(c)|}, \] 
is bounded from above and away from zero.

In this paper, a map 
$f \colon [0,1] \longrightarrow [0,1]$ is called $C^r$, $r\ge 1$,
\emph{unimodal} if the following holds:

\begin{itemize}
    \item There exists $c\in (0,1)$ such that $f$ is $C^r$ and $f'\not=0$ on $[0,1]\setminus \{c\}$;
    \item there exist real numbers $\ell_-,\ell_+\ge 1$, a neighborhood $U$ of
    $c$, and $C^r$ diffeomorphisms $\phi_-,\phi_+, \varphi_-, \varphi_+$ of $\R$
    with $\phi_{\pm}(c) = \varphi_{\pm}(f(c)) = 0$ such that  
    \[ |\varphi_- \circ f(x)| = |\phi_-(x)|^{\ell_-}, \] for every $x \in U$ with $x \le c$, and 
    \[ |\varphi_+ \circ f(x)| = |\phi_+(x)|^{\ell_+}, \] for every $x \in U$ with $x \ge c$.
\end{itemize}
Let $\mathcal{U}^r$ denote the collection of $C^r$ unimodal maps.
If $\ell_- = \ell_+$ we say that the map is \emph{(weakly) symmetric}, and 
if $\ell_- \neq \ell_+$ we say that the map is \emph{(strongly) asymmetric}.

It is a long-standing open question whether an asymmetric unimodal map has a wandering interval, see
\cite[Chapter IV.11]{MevS11}.
Recently, Kozlovski and van Strien \cite{Kozlovski_van_Strien_Asymmetric_2020} proved the absence of 
wandering intervals for asymmetric S-unimodal maps with $2^{\infty}$ combinatorics 
and $1 = \ell_- \le \ell_+$ . Their method relies in an essential way on the specific 
combinatorics, and the assumption that one of the critical orders is equal 
to one.

In this paper, we shall prove the following theorem. 

\begin{namedtheorem*}{Main Theorem}
    Each $f\in \mathcal{U}^2$ has no wandering interval.
\end{namedtheorem*}

Recall that a {\em wandering interval} of $f$ is an interval $J\subset [0,1]$ such that 
\begin{itemize}
\item the intervals $J$, $f(J)$, $f^2(J)$, $\cdots$, are pairwise disjoint;
\item for each $x\in J$, $\omega(x)$ is not a periodic orbit. 
\end{itemize}

A well known consequence of non-existence of wandering intervals is the following contraction principle:
for any $\e>0$ there exists $\delta>0$ such that if $I$ is an interval with $|I|<\delta$ and disjoint from
the immediate basin of the periodic attractors, then for any $n \ge 1$, any component $J$ of $f^{-n}(I)$ 
has length less than $\e$. This contraction principle plays important roles in one-dimensional dynamics. 
For example, it implies the following result.

\begin{coro}
    If $f$ is a $C^2$ (asymmetric) unimodal map, then there exists 
    $\lambda > 1$ and $n_0 \in \N$ so that \[ |Df(p)|\ge \lambda \] for every 
    periodic point $p$ of period $n \geq n_0$.
\end{coro}
For a proof of this corollary see 
\cite{Martens_de_Melo_van_Strien_Julia_Fatou_Sullivan}, 
\cite[Chapter 4]{MevS11}.

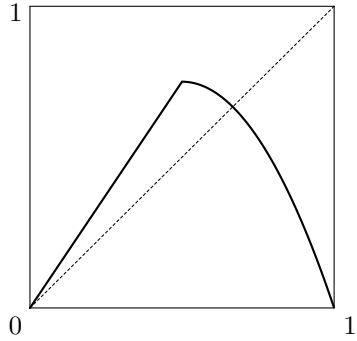
\begin{figure}
    \centering
    \begin{tikzpicture}[line cap=round,line join=round,>=triangle 45,x=2cm,y=2cm]
      \clip(-1.3,-1.3) rectangle (1.3,1.3);
      \draw (1,1)-- (-1,1);
      \draw (-1,1)-- (-1,-1);
      \draw (-1,-1)-- (1,-1);
      \draw (1,-1)-- (1,1);
      \draw [dash pattern=on 1pt off 1pt] (-1,-1)-- (1,1);
      \draw (-1.2,1.08) node[anchor=north west] {1};
      \draw (1,-1.) node[anchor=north west] {1};
      \draw (-1.2,-1.) node[anchor=north west] {0};
      \draw [line width=0.8] (-1,-1)--(0,0.5);
      \draw[line width=0.8,smooth,samples=100,domain=.0:1.0] plot(\x,{1.5*(1-(\x)^2)-1});
    \end{tikzpicture}
    \caption{Graphic of an asymmetric unimodal map}
    \label{fig:enter-label}
\end{figure}

\subsection*{Acknowledgments}
The work is supported by the National R\&D Program 2021YFA1003200.
WS is also supported by the New Cornerstone Science Fundation through the New Cornerstone Investigator Program and the XPLORER PRIZE.

\section{Proof of the Main Theorem} 
\label{section:unimodal_case}
In this section, we shall prove the Main Theorem, suspending the proof of a few technical lemmas to the next section.  

Throughout the rest of this work, we will use $\N$ to denote the set of
positive integers and $\N_0$, to denote the set of non-negative integers,
thus $\N_0 \= \N \cup \{ 0 \}$. 
For an interval $J \subset \R$ we denote its length by $|J|$.

Let $f\in\mathcal{U}^2$ and let $c\in (0,1)$ be its critical point.
Let $V$ be a small neighborhood of the critical point $c$ so that 
$f(V) \subset [c,f(c)]$. Given $x \in V \setminus \{c\}$, we denote by 
$\tau(x)$ the point in $V$ so that $f(\tau(x)) = f(x)$ and $\tau(x)
\neq x$. We write $c_n=f^n(c)$. 

The Main Theorem will be proved by contradiction. Assume that $f$ has a wandering interval $Q$. We may 
assume that $Q$ is a maximal wandering interval of $f$ in the sense that it is not strictly contained in 
another wandering interval. Let $$Q_n=f^n(Q).$$ 
Clearly, $|Q_n|\to 0$ as $n\to\infty$. As a consequence of ~\cite{Schwartz_A1963_non-existence_of_wandering_intervals}, there is a sequence $n_k\to\infty$ such that $f^{n_k}(Q)\to c$, see~\cite[Theorem I.2.2]{MevS11} or~\cite[Lemma 4.1]{deMelo-vanStrien1989_Structure_theorem}.   
It follows that $$\overline{Q}_n\not\ni c,\,\forall n\in \mathbb{N}_0.$$

\begin{defi}
    \label{def:closest_return_moment}
    We say that $n\ge 1$ is {\em a closest return moment (to the critical value)} if 
    \[ \dist(Q_n, c_1)< \dist (Q_k, c_1)\] for each $0\le k<n$. 
\end{defi}
Let $s(0)=0$ and let $$s(1)<s(2)<\cdots$$ be all the closest return moments. Then
$$Q_{s(k)}\to c_1 \text{ as } k\to\infty.$$ 
For each $k \in \N_0$, we call
$Q_{s(k)}$ the \emph{$k$th closest return of $Q$ to $c_1$}.
For $k \in \N_0$, let $a_k, b_k$ denote the left and
right endpoint of $Q_{s(k)}$ respectively. 
Put 
\[ \eta_k:=\frac{|b_k-c_1|}{|a_k-c_1|}.\] 

The following lemma can be proved in exactly the same way as the symmetric case. 

\begin{namedtheorem*}{Lemma A}(\cite{Guckenheimer1979_sensitive_dependence}, \cite{Blokh-Lyubich_non-existence_of_wandering_intervals}, \cite{Martens_de_Melo_van_Strien_Julia_Fatou_Sullivan})
    \label{lem:etan0} 
    If $f$ has a wandering interval $Q$, then \[ \eta_k \to 0, \text{ as }k \to \infty.\]
\end{namedtheorem*}

The following lemma is the main step of our argument. It shows that, for $k$ large enough, pre-closest 
returns $Q_{s(k)-1}$, are all on the same side with respect to the turning point $c$. The side they lie is
the one corresponding to the larger critical order. 

\begin{namedtheorem*}{Lemma B} 
    \label{lem:position_of_Q_s(k)}
    If $f$ has a wandering interval $Q$, then $\ell_-\not=\ell_+$ and there exists a constant $\alpha>0$ such that one of the following holds:
    \begin{enumerate}
        \item $\ell_-<\ell_+$; for all $k$ large, $Q_{s(k)-1}$ lies to the right hand side of $c$; and 
        $$\liminf_{k\to\infty} \frac{\eta_{k-1}^\alpha|a_k-c_1|}{|a_{k-1}-c_1|^{\ell_+}}>0,$$  
        \item $\ell_+<\ell_-$; for all $k$ large, $Q_{s(k)-1}$ lies to the left hand side of $c$; and 
        $$\liminf_{k\to\infty} \frac{\eta_{k-1}^\alpha|a_k-c_1|}{|a_{k-1}-c_1|^{\ell_-}}>0.$$
    \end{enumerate}
\end{namedtheorem*}

We need also the following lemma. 
\begin{namedtheorem*}{Lemma C} 
    \label{lem:decreasing_estimates_eta}
    There exists a constant $C>0$ such that for all $k$ large, 
    $$\eta_k\le C \eta_{k-1}^{\max(\ell_+,\ell_-)}.$$
\end{namedtheorem*}

\begin{proof}[Completion of proof of the Main Theorem]
To be definite, we assume that $\ell_-\le \ell_+$. Then by Lemma B, $\ell_-<\ell_+$. 

By Lemma C, there is a constant $C_1>0$ such that 
\begin{equation}\label{eqn:etaspeed}
    \eta_k\le e^{-C_1\ell_+^k}
\end{equation}
holds for all $k$ large. Since $\eta_{k-1}\to 0$, by Lemma B, we have $$|a_k-c_1|>|a_{k-1}-c_1|^{\ell_+}$$
for all $k$ large. Thus there is $C_2>0$ such that 
\begin{equation}\label{eqn:akspeed0}
    |a_k-c_1|\ge e^{-C_2\ell_+^k}
\end{equation}
holds for all $k$ large. 

By (\ref{eqn:etaspeed}) and (\ref{eqn:akspeed0}), there is $\beta>0$ such that 
$$\lim_{k\to\infty} \frac{\eta_{k-1}^\alpha}{|a_{k-1}-c_1|^\beta}= 0.$$
Applying Lemma B again, we have 
$$|a_k-c_1|\ge |a_{k-1}-c_1|^{\ell_+-\beta},$$
holds for all $k$ large.
Thus
$$|a_k-c_1|\ge e^{-C_3(\ell_+-\beta)^k}.$$ But
\begin{equation}\label{eqn:etavsa}
    |a_{k}-c_1|\le |b_{k-1}-c_1|\le \eta_{k-1}\le e^{-C_1\ell_+^{k-1}},
\end{equation}
a contradiction.
\end{proof}


\section{Proof of Lemmas}
In this section, we prove Lemmas A, B and C. Throughout, we fix $f\in \mathcal{U}^2$.

\subsection{Cross-ratio distortion and the Koebe principle}
\label{Section_cross-ration_distortion}
We shall use the standard cross-ratio estimates and the Koebe principle. 

Let $J \subset T$ be intervals such that $T \setminus J$ has two connected components $L$ and $R$. 
We consider the following \emph{cross-ratio} of these intervals
\begin{equation}
    \label{eq:cross-ratio_of_intervals}
    D(T,J) \= \frac{|T|\, |J|}{|L| \, |R|}.
\end{equation}
If $g:T\to g(T)$ is a homeomorphism, then we define the \emph{cross-ratio distortion of $g$} as 
\begin{equation}
    \label{eq:cross-ratio_distortion}
    B(g,T,J) \= \frac{D(g(T),g(J))}{D(T,J)}.
\end{equation}

We shall often use the following observation.

\begin{lemm}\label{lem:onesidebig}
Let $J \subset T$ be intervals such that $T \setminus J$ has two
connected components $L$ and $R$. Suppose that $|L|\ge |J|$. Then 
$$D(T,J)\le \frac{2|J|}{|R|}+1.$$
\end{lemm}
\begin{proof} Let $x=|L|/|J|$ and $y=|R|/|J|$. Then $x\ge 1$, hence 
$$D(T,J)=\frac{1+x+y}{xy}\le \frac{2+y}{y}=\frac{2|J|}{|R|}+1.$$
\end{proof}

The following results are well-known and can be found in for example \cite[Chapter IV]{MevS11}
\begin{theo}\label{theo:CR_distortion_bound}
    Let $f \in\mathcal{U}^2$. Then there exists a constant $C_0 = C_0(f)>0$ such that if 
    $J \subset T \subset [0,1]$ are intervals with $T \setminus J$
    consisting of two connected components, and $n$ is a positive integer such that 
    $f^n|T$ is monotone, then
    \begin{equation}
        \label{eq:CR_distortion_bound}
        \log B(f^n,T,J) \geq -C_0\sum_{j=0}^{n-1}|f^j(T)|.
    \end{equation}
\end{theo}

We say that an interval $T$ contains a $\tau$-scaled neighborhood of another interval $J$ if $T\setminus J$ has two components and each component has length at least $\tau |J|$. 
\begin{theo}[One-sided Koebe principle]
    \label{lemm:minimum_principle}
    For each $C>0$, $\rho>0$ there exists $K<\infty$ with the following property. Let $T=[u,v]$ and let $g:T\to g(T)$ be a $C^1$ diffeomorphism. Assume that for any intervals $J_*\subset T_*\subset T$ such that $T_*\setminus J_*$ has two connected components, we have 
     $$B(g, T_*, J_*)\ge C.$$ Then $$|Dg(w)|\le K |Dg(v)|$$
     for each $w\in [u,v]$ with 
     $$\frac{|g(w)-g(v)|}{|g(T)|}\ge \rho.$$
\end{theo}

\subsection{Proof of Lemma A}
This section is devoted to the proof of Lemma A. In the case that $f$ has negative Schwarzian, this is 
proved in \cite{Guckenheimer1979_sensitive_dependence}. For general $f\in \mathcal{U}^2$, this is 
essentially proved in \cite{Blokh-Lyubich_non-existence_of_wandering_intervals} and 
\cite{Martens_de_Melo_van_Strien_Julia_Fatou_Sullivan}. We include a detailed proof for the reader's convenience.

For $k \in \N$, put $$T_k=(a_{k-1}, c_1)\supset \widehat{T}_k=(a_{k-1}, b_{k+1}).$$ For
$0 \leq j \le s(k)$, let $T_k^j$ (resp. $\widehat{T}_k^j$) denote the connected component of 
$f^{-(s(k)-j)}(T_k)$ (resp. $f^{-(s(k)-j)}(\widehat{T}_k)$) which contains $Q_j$. 

\begin{lemm}
    \label{lem:T_k}
    For every $k \in \N$, $f^{s(k)}: T_k^0\to T_k$ is a
    diffeomorphism and $T_k^0\not\ni c_1$.  Moreover, there exists a constant $M>0$ such that  
    \begin{equation}\label{eqn:hattk}
    B(f^{s(k)}, \widehat{T}_k^0, Q)\ge M,
    \end{equation}
    for each $k\ge 1$. 
\end{lemm}

\begin{proof} Let us prove the first statement by contradiction. 
So suppose that this is not true. Then there exists a maximal $0 < n< s(k)$ such
that $c_1 \in T_k^n$. Since $s(k-1)$ and $s(k)$ are closest return moments,
either $Q_n=Q_{s(k-1)}$ or $Q_n$ lies to the left of $Q_{s(k-1)}$, so $T_k \subset T_k^n$. Thus
$f^{s(k)-n}$ maps $\overline{T_k^n}$ homeomorphically into itself, which 
implies that for each $x\in T^n_k$, $\omega(x)$ is a periodic orbit. This is a contradiction, since 
$T^n_k$ contains the wandering interval $Q_n$.

To prove (\ref{eqn:hattk}), let $L, R\subset \hT_k^0$ be such that $f^{s(k)}(L)=(a_{k-1}, a_k)$ and $f^{s(k)}(R)=(b_k, b_{k+1})$. By Theorem~\ref{theo:CR_distortion_bound}, it suffices to show that 
\begin{equation}\label{eqn:LA}
    \sum_{j=0}^{s(k)-1}|f^j(L)|\le 2
\end{equation}
and 
\begin{equation}\label{eqn:RA}
    \sum_{j=0}^{s(k)-1}|f^j(R)|\le 1,
\end{equation}
since $$\sum_{j=0}^{s(k)-1} |f^j(\hT_k^0)|= \sum_{j=0}^{s(k)-1}|f^j(L)|+\sum_{j=0}^{s(k)-1}|f^j(R)|+\sum_{j=0}^{s(k)-1} |Q_j|,$$
and $$\sum_{j=0}^{s(k)-1}|Q_j|\le 1.$$

To prove (\ref{eqn:LA}), it suffices to show that 
if $f^{j_1}(L)\cap f^{j_2}(L)\not=\emptyset$ for some $0\le j_1<j_2< s(k)$, then $j_2-j_1=s(k)-s(k-1)$.   
Let $j=s(k)-j_2+j_1$. Arguing by contradiction, suppose there exists $y\in f^j(L)\cap (a_{k-1}, b_k)$. Assume that $j\not =s(k-1)$. Then $Q_j$ lies to the left of $T_k$, hence 
$f^j(L)\cap Q_{s(k-1)}\not=\emptyset$.
So $$f^{s(k)}(L)\cap Q_{s(k)-j+s(k-1)}\not=\emptyset,$$
which implies that $s(k)-j+s(k-1)>s(k)$, i.e. $j<s(k-1)$. Let $L':=f^{s(k)-s(k-1)}(f^j(L))$. Then
$L'\supset Q_{s(k)-s(k-1)+j}$ and 
$$L'\cap Q_{s(k)}\supset f^{s(k)-s(k-1)}(f^j(L)\cap Q_{s(k-1)})\not=\emptyset.$$
Since $Q_{s(k)-s(k-1)+j}$ lies to the left of $a_{k-1}$, it follows that $L'\supset f^{s(k)}(L)$. Thus $f^{s(k-1)-j}$ maps $L'$ homoemorphically into $L'$. 
But this implies that for any $x\in L'$, $\omega(x)$ is a periodic orbit, which is in contradiction with the fact that $L'$ contains a wandering interval. We have completed the proof of (\ref{eqn:LA}).

To prove (\ref{eqn:RA}), we shall show that $f^j(R)$, $0\le j<s(k)$, are pairwise disjoint. Arguing by contradiction, assume that there exists $0\le j_1<j_2<s(k)$ such that 
$f^{j_1}(R)\cap f^{j_2}(R)\not=\emptyset$. Then for $j=s(k)-j_2+j_1$, $f^j(R)\cap (b_j, b_{j+1})$ contains 
a point $y$. Since $f^j(R)$ has an endpoint in $\partial Q_j$ which lies to the left of $a_k$, we have 
$f^j(R)\supset \overline{Q_{s(k)}}$. Thus $f^{s(k)}(R)=f^{s(k)-j}(f^j(R))\supset \overline{Q_{2s(k)-j}}$. 
Therefore $2s(k)-j>s(k+1)$. But then $f^{s(k+1)-s(k)}(R)$ contains $Q_{s(k+1)}$ and has an endpoint in 
$\partial Q_{s(k+1)-s(k)+j}$ which lies to the left of $b_k$. It follows that 
$I:=f^{s(k+1)-s(k)}(R)\supset (b_k, b_{k+1}).$ Thus $f^{2s(k)-j-s(k+1)}$ maps $I$ homeomorphically into
itself and a contradiction arises as before.    
\end{proof}

\begin{proof} [Proof of Lemma A]
By Lemma \ref{lem:T_k},
for $k \in \N$, we can put $L_k$ and $R_k$ to be the connected components of 
$\widehat{T}_k^0 \setminus Q$ such that 
$f^{s_k}(L_k)=(a_{k-1}, a_k)$ and $f^{s(k)}(R_k)=(b_k, b_{k+1})$. Since $Q$ 
is a maximal wandering interval, we have 
\begin{equation}\label{eqn:epsk}
    |R_k|/|Q|\to 0,
\end{equation}
as $k \to \infty$. By Lemma~\ref{lem:T_k}, $$B(f^{s(k)}, \widehat{T}_k^0, Q) \ge M,$$ i.e.
$$D(\widehat{T}_k, Q_{s(k)})\ge M D(\widehat{T}_k^0, Q).$$
Since $$D(\widehat{T}_k^0, Q)=\frac {|\widehat{T}_k^0||Q|}{|L_k||R_k|}\ge \frac{|Q|}{|R_k|},$$
we have 
\begin{equation}\label{eqn:TkCR}
    D(\widehat{T}_k, Q_{s(k)})\to \infty
\end{equation}
as $k\to\infty$, by (\ref{eqn:epsk}).

Let us prove that there exists $k_0$ such that $\delta_k:=b_k-a_k$ decreases to $0$ for $k\ge k_0$. 
Otherwise, we can find an arbitrarily large $k$ such that $\delta_k\le \delta_{k-1}$ and 
$\delta_k\le \delta_{k+1}$. For such $k$, both components of $\widehat{T}_k\setminus Q_{s(k)}$ have length 
not smaller than $|Q_{s(k)}|$, so by Lemma~\ref{lem:onesidebig}, $D(\widehat{T}_k, Q_{s(k)})\le 3$, 
a contradiction.  

In particular, $|a_{k-1}-a_k|\ge |a_k-b_k|$ holds for all $k$ large enough. So the left component of $\widehat{T}_k\setminus Q_{s(k)}$ has length at least $|Q_{s(k)}|$. By 
Lemma~\ref{lem:onesidebig}, we obtain 
(\ref{eqn:TkCR}), this implies that
$$D(\widehat{T}_k, Q_{s(k)})\le \frac{2|a_k-b_k|}{|b_k-b_{k+1}|}+1\le \frac{2|b_{k-1}-b_k|}{|b_k-b_{k+1}|}+1.$$
By (\ref{eqn:TkCR}), 
$$\frac{|b_k-b_{k+1}|}{|b_{k-1}-b_k|}\to 0.$$
Therefore, $\eta_k\to 0$. 
\end{proof}

\subsection{Proof of Lemma B}\label{subsec:lemB}
Let $a_k', b_k'$ denote the endpoints of $Q_{s(k)-1}$ such that
$$f(a_k')=a_k,\, f(b_k')=b_k.$$ Put $B_k \= f^{-1}((a_{k-1}, c_1])$. For all
$k$ large enough, $B_k$ is the open interval bounded by $a'_{k-1}$ and $\tau(a'_{k-1})$, the two points in
$f^{-1}(a_{k-1})$. For $0 \leq j < s(k)$, let $B_k^j$ denote the connected
component of $f^{-(s(k)-1-j)}(B_k)$ that contains $Q_j$. So $B_k=B_k^{s(k)-1}$.

\begin{lemm}\label{lem:Bk}
    For all $k$ sufficiently large, the following hold:
    \begin{enumerate}
        \item[(1)] $f^{s(k)-s(k-1)-1}: B_k^{s(k-1)}\to B_k$ is a homeomorphism;
        \item[(2)] $a_{k-2}\not\in B_k^{s(k-1)}$.
        \item[(3)] The intervals $B_k^j$, $s(k-1)\le j<s(k)-1$ are pairwise disjoint, hence 
        $$\sum_{j=s(k-1)}^{s(k)-2}|B_k^j|\le 1.$$
    \end{enumerate}
\end{lemm}
\begin{proof}
Let us first prove (1). Suppose that it is not true.
Thus, there exists a maximal $s(k-1) \le j <s(k)-1$ such that $B_k^j$ 
contains $c$. Since $B_k^{j+1}$ contains $c_1$ and $Q_{j+1}$, and $s(k-1)<j+1<s(k)$, we must have 
\[ \dist (Q_{j+1},c_1) > \dist (Q_{s(k-1)},c_1). \] Thus, 
$[a_{k-1},c_1] \subset B_k^{j+1}$, and hence $B_k \subsetneq B_k^j$. This
implies that $Q_{s(k-1)-1} \subset B_k^j$, and applying $f^{s(k)-j}$
we get \[Q_{s(k)+s(k-1)-1-j} \subset B_k.\] Thus, 
$Q_{s(k)+s(k-1)-j} \subset (a_{k-1}, c_1]$. Since \[ s(k)+s(k-1)-j
< s(k)\] this is a contradiction.

Let us now prove (2). Suppose that $a_{k-2} \in B_k^{s(k-1)}$. Then
$\overline{Q_{s(k-2)}}\subset B_k^{s(k-1)}$. Consequently,
\[Q_{s(k)-s(k-1)+s(k-2)}=f^{s(k)-s(k-1)}(Q_{s(k-2)})\subset 
f(B_k)=(a_{k-1}, c_1].\] Since $s(k)-s(k-1)+s(k-2) < s(k)$, this cannot 
happen.

Finally, let us prove (3). It suffices to show that the intervals $B_k^j$, $s(k-1)\le j<s(k)-1$, are 
pairwise disjoint. 
Arguing by contradiction, assume that there exist $s(k-1)\le j_1<j_2<s(k)-1$ such that 
$B_k^{j_1}\cap B_k^{j_2}$ is non-empty. Then $B_k\cap B_k^j$ is non-empty, where $j=s(k)-1-j_2+j_1$.
Since $s(k-1)<j<s(k)-1$, $Q_{j+1}$ lies to the left of $f(B_k)$. Since $f(B_k^j)\supset Q_{j+1}$ and it 
intersects $f(B_k)$, we obtain that
$f(B_k^j)$ intersects the interior of $Q_{s(k-1)}$. 
Thus $f(B_k)= f^{s(k)-j}(B_k^j)$ intersects the interior of $Q_{s(k-1)+s(k)-j}$. Since 
$s(k-1)<s(k-1)+s(k)-j<s(k)$, this is impossible.
\end{proof}

\begin{figure}
    \centering
    \begin{tikzpicture}[line cap=round,line join=round,>=triangle 45,x=1.6cm,y=1.6cm]
      \clip(-5.3,-1.9) rectangle (4.3,1.9);
      \draw (-0.3,0.5)-- (-5,0.5);
      \draw (4,0.5)-- (0.5,0.5);
      \draw (1.5,-1)-- (2.9,-1);
      \draw (-5.2,0.9) node[anchor=north west] {$\tau(a_{k-1}'')$};
      \draw (-5,0.4) -|(-5,0.6);
      \draw (-4.15,0.9) node[anchor=north west] {$c$};
      \draw (-4,0.4) -|(-4,0.6);
      \draw (-3.45,0.95) node[anchor=north west] {$b_k'$};
      \draw (-3.3,0.4) -|(-3.3,0.6);
      \draw (-2.55,0.95) node[anchor=north west] {$a_k'$};
      \draw (-2.4,0.4) -|(-2.4,0.6);
      \draw (-2.15,0.95) node[anchor=north west] {$b_{k-1}'$};
      \draw (-2,0.4) -|(-2,0.6);
      \draw (-0.55,0.95) node[anchor=north west] {$a_{k-1}''$};
      \draw (-0.45,0.4) -|(-0.45,0.6);
      \draw (-5,0.54)-- (-0.45,0.54);
      \draw [dash pattern=on 1pt off 1pt] (-3.3,0.46)-- (-2.4,0.46);
      \draw [dash pattern=on 1pt off 1pt] (-2,0.46)-- (-0.45,0.46);
      \draw (-3,0.4) node[anchor=north west] {$B_k$};
      \draw (-3.25,1.2) node[anchor=north west] {$Q_{s(k)-1}$};
      \draw (-1.75,1.2) node[anchor=north west] {$Q_{s(k-1)-1}$};
      \draw (0.55,0.95) node[anchor=north west] {$a_{k-2}$};
      \draw (0.65,0.4) -|(0.65,0.6);
      \draw (1.45,0.95) node[anchor=north west] {$a_{k-1}$};
      \draw (1.65,0.4) -|(1.65,0.6);
      \draw (2.15,0.95) node[anchor=north west] {$b_{k-1}$};
      \draw (2.45,0.4) -|(2.45,0.6);
      \draw (1.45,-1.1) node[anchor=north west] {$a_{k-1}$};
      \draw (1.65,-1.1) -|(1.65,-0.9);
      \draw (2.15,-1.1) node[anchor=north west] {$b_{k-1}$};
      \draw (2.45,-1.1) -|(2.45,-0.9);
      \draw (1.6,-0.5) node[anchor=north west] {$Q_{s(k-1)}$};
      \draw (3,-0.8) node[anchor=north west] {$B_k^{s(k-1)}$};
      \draw [dash pattern=on 1pt off 1pt] (1.65,-0.97) --(2.45,-0.97);
      \draw [dash pattern=on 1pt off 1pt] (1.65,0.46)-- (2.45,0.46);
      \draw (1.65,0.4) node[anchor=north west] {$Q_{s(k-1)}$};
      \draw (2.75,0.95) node[anchor=north west] {$a_{k}$};
      \draw (2.95,0.4) -|(2.95,0.6);
      \draw (3.25,0.95) node[anchor=north west] {$b_{k}$};
      \draw (3.35,0.4) -|(3.35,0.6);
      \draw [dash pattern=on 1pt off 1pt] (2.95,0.46)-- (3.35,0.46);
      \draw (2.9,0.4) node[anchor=north west] {$Q_{s(k)}$};
      \draw (3.85,0.95) node[anchor=north west] {$c_1$};
      \draw (4,0.4) -|(4,0.6);
    \end{tikzpicture}
    \caption{On top, the double lined interval denote $B_k$. On the bottom, 
    we have the interval $B_k^{s(k-1)}$. Dashed lines represent closest 
    return of $Q$ to $c_1$, and pre-closest return around $c$.}
    \label{fig:B_k_graph}
\end{figure}
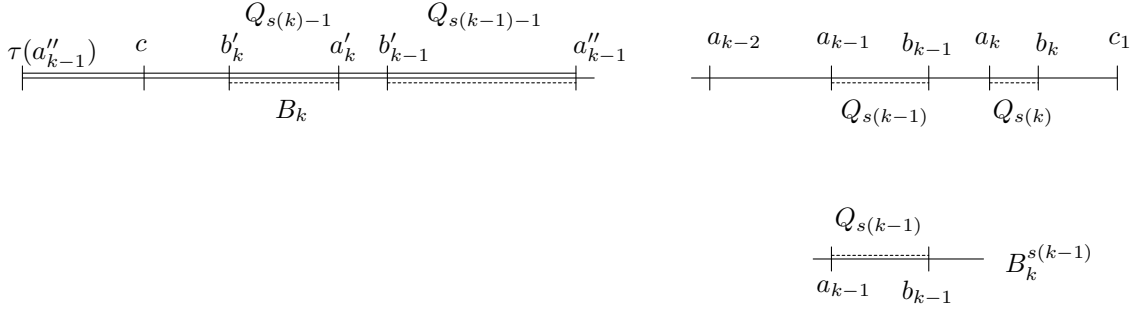

\begin{proof}[Proof of Lemma B] 
We shall only deal with the case $\ell_-\le \ell_+$ as the other case is similar. Let $k$ be a large 
integer. We shall first show that $Q_{s(k)-1}$ lies to the right hand side of $c$. 
Let $a_{k-1}''$ (resp. $b''_{k-1}$) denote the point in $f^{-1}(a_{k-1})$ (resp. $f^{-1}(b_{k-1})$) which 
is in the same side of $c$ as $a_k'$. We distinguish two cases. 

\noindent \textbf{ Case 1.} $B_k^{s(k-1)}\not \ni c_1$. 
Then
$$D(B_k, Q_{s(k)-1})^{-1}=\frac{|\tau(a''_{k-1})-b_k'|}{|\tau(a''_{k-1})-a''_{k-1}|}
\frac{|a''_{k-1} - a_k'|}{|a'_k-b'_k|}\ge \frac{|\tau(a''_{k-1})-c|}{|\tau(a''_{k-1})-a''_{k-1}|}\frac{|a''_{k-1}-a'_k|}{|a'_k-c|}.$$
Denote by $R$ be the right component of $B_k^{s(k-1)}\setminus Q_{s(k-1)}$. Then 
$$\eta_{k-1}\ge \frac{|R|}{|Q_{s(k-1)}|}\ge D(B_k^{s(k-1)}, Q_{s(k-1)})^{-1}.$$
By Lemma~\ref{lem:Bk} and Theorem~\ref{theo:CR_distortion_bound}, we have 
\begin{equation}\label{eqn:Dbk}
D(B_k,Q_{s(k)-1})\ge M\cdot D(B_k^{s(k-1)}, Q_{s(k-1)}),
\end{equation}
where $M>0$ is a constant. Thus 
\begin{equation}\label{eqn:ak-1ak}
    \eta_{k-1} \ge M \frac{|\tau(a''_{k-1})-c|}{|\tau(a''_{k-1})-a''_{k-1}|}\frac{|a''_{k-1}-a'_k|}{|a'_k-c|}
\end{equation}

By Lemma A and the local behavior of $f$ near $c$, $|a''_{k-1}-a'_k|/|a'_k-c|$ is large, hence 
\begin{equation}
\frac{|\tau(a_{k-1}'')-c|}{|\tau(a''_{k-1})-a''_{k-1}|}\le \frac{\eta_{k-1}}{2}
\end{equation}
is small, provided that $k$ is large enough, again by Lemma A. 
It follows that 
that $a_k'$ lies to the right of $c$.

\noindent \textbf{Case 2.} $B_k^{s(k-1)}\ni c_1$. 
Put $H_k \=(u,c_1)$, where $u$ is the left endpoint of $B_k^{s(k-1)}$. 
Let $$x=c_{s(k)-s(k-1)},\text{ and } g=f^{s(k)-s(k-1)-1}.$$ 
By Lemma~\ref{lem:T_k}, $g$ maps $T_k^{s(k-1)}$ homeomorphically onto 
$T_k^{s(k)-1}$ and $T_k^{s(k-1)}\not\ni c_1$. Since $T_k^{s(k)-1}$ is the component of $B_k\setminus \{c\}$ 
containing $Q_{s(k)-1}$ and $T_k^{s(k-1)}\supset Q_{s(k-1)}$, we have $T_k^{s(k-1)}\subset H_k$. Therefore, 
$x$ and $Q_{s(k)-1}$ lies in different sides of $c$, and  
\begin{equation}\label{eqn:bk-1bk}
    g(b_{k-1})=b_k',  \, f^{s(k)-s(k-1)}(b_{k-1})=b_k, \text{ and }a''_{k-1}=g(u).
\end{equation}

Let us show that there exists $y\in [b_{k-1}, c_1]$ such that 
\begin{equation}\label{eqn:Dfy}
|Df^{s(k)-s(k-1)}(y)|\ge \frac{1}{2}.
\end{equation}

Arguing by contradiction, assume that this is false.  Then for any $y\in [b_{k-1}, c_1]$, 
$$|f^{s(k)-s(k-1)}(y)-b_k|=|f^{s(k)-s(k-1)}(y)-f^{s(k)-s(k-1)}(b_{k-1})|\le \frac{1}{2}|b_{k-1}-c_1|,$$
where we use (\ref{eqn:bk-1bk}) for the equality. 
Note that 
$f^{s(k)-s(k-1)}(y)\le c_1$ for each $y \in [b_{k-1},c_1]$. By Lemma A, $|b_k-c_1|/|b_{k-1}-c_1|<\eta_k$ is
small. Thus $f^{s(k)-s(k-1)}$ maps $[b_{k-1}, c_1]$ into itself as a contraction. Thus for any 
$y\in [b_{k-1}, c_1]$, $\omega(y)$ is a periodic point. This is in contradiction with the fact that 
$[b_{k-1},c_1]$ contains a wandering interval $Q_{s(k)}$. 

By Lemma~\ref{lem:Bk} and Theorem~\ref{theo:CR_distortion_bound},
$$B(g, H_k, Q_{s(k-1)})\ge M,$$
where $M>0$ is a constant. 
Thus, 
\begin{multline*}
D((a''_{k-1},x), Q_{s(k)-1})=\frac{|a_{k-1}''-x||a_k'-b_k'|}{|a_{k-1}''-a_k'||x-b_k'|}\ge M \frac{|H_k||a_{k-1}-b_{k-1}|}{|u-a_{k-1}||b_{k-1}-c_1|}\\
\ge M \frac{|a_{k-1}-b_{k-1}|}{|b_{k-1}-c_1|}=M(\eta_{k-1}^{-1}-1)\ge \frac{M}{2\eta_{k-1}}+1
\end{multline*}
since $\eta_{k-1}$ is small by Lemma A. Again by Lemma A, $|a'_k-b'_k|\le |a_{k-1}''-a_k'|$. So by Lemma~\ref{lem:onesidebig},
$$D((a''_{k-1},x), Q_{s(k)-1})\le \frac{2|a'_k-b'_k|}{|b_k'-x|}+1.$$
Thus
$$|x-b_k'|\le \frac{4}{M}\eta_{k-1} |a_k'-b_k'|. $$
In particular, for each $y'\in [b_k',x]$, we have 
\begin{equation}\label{eqn:xc}
|y'-c|\le \frac{4}{M}\eta_{k-1}|a_k'-c|<|a'_k-c|.
\end{equation}

Let us prove that $Q_{s(k)-1}$ lies to the right of $c$. Arguing by contradiction, assume the contrary. 
Then, by (\ref{eqn:xc}) and Lemma A, both components of  $B_k\setminus [a_k',x]$ have length larger than $|b_k'-x|$. 
By Lemma~\ref{lem:Bk} and Theorem~\ref{theo:CR_distortion_bound}, for any intervals $J_*\subset T_*\subset H_k$, 
$$B(f^{s(k)-s(k-1)-1}, T_*, J_*)\ge M.$$
So by the one-sided Koebe principle, 
\begin{equation}\label{eqn:Dgy}
|Dg(y)|\le C\frac{|a_k'-b'_k|}{|a_{k-1}-b_{k-1}|}\le C'\frac{|a_k'-c|}{|a_{k-1}-c|}
\end{equation}
for each $y\in [b_{k-1}, c_1]$.
Since $y':=g(y)\in [b_k', x]$, we have 
$$|Df(y')|\le A |y'-c|^{\ell_--1}\le A |a_k'-c|^{\ell_--1},$$
where $A>0$  is a constant, and the last inequality follows from (\ref{eqn:xc}).
Consequently, 
\begin{equation}\label{eqn:Dfyy}
|Df^{s(k)-s(k-1)}(y)|\le C'' \frac{|a_k'-c|^{\ell_-}}{|a_{k-1}-c_1|}\le C''\frac{|a_k-c_1|}{|a_{k-1}-c_1|}<\frac{1}{2},
\end{equation}
where the last inequality follows from Lemma A (assuming $k$ large). This is in contradiction with (\ref{eqn:Dfy}).

So far, we have shown that if $\ell_-\le \ell_+$, then $Q_{s(k)-1}$ lies to the right of $c$ when $k$ is large enough. 
If $\ell_-\ge \ell_+$, then a similar argument shows that $Q_{s(k)-1}$ cannot
lie in the left hand side of $c$ when $k$ is large enough. 
Thus $\ell_-\le \ell_+$ implies that $\ell_-<\ell_+$.

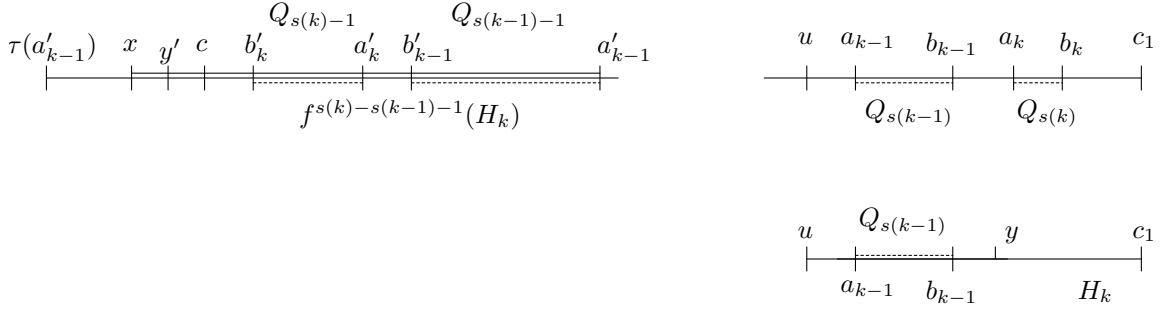
\begin{figure}
    \centering
    \begin{tikzpicture}[line cap=round,line join=round,>=triangle 45,x=1.6cm,y=1.6cm]
      \clip(-5.3,-1.9) rectangle (4.3,1.9);
      \draw (-0.3,0.5)-- (-5,0.5);
      \draw (4,0.5)-- (0.9,0.5);
      \draw (1.5,-1)-- (2.9,-1);
      \draw (-5.4,0.95) node[anchor=north west] {$\tau(a_{k-1}')$};
      \draw (-5,0.4) -|(-5,0.6);
      \draw (-4.45,0.9) node[anchor=north west] {$x$};
      \draw (-4.3,0.4) -|(-4.3,0.6);
      \draw (-4.15,0.9) node[anchor=north west] {$y'$};
      \draw (-4.0,0.4) -|(-4.0,0.6);
      \draw (-3.85,0.9) node[anchor=north west] {$c$};
      \draw (-3.7,0.4) -|(-3.7,0.6);
      \draw (-3.45,0.95) node[anchor=north west] {$b_k'$};
      \draw (-3.3,0.4) -|(-3.3,0.6);
      \draw (-2.55,0.95) node[anchor=north west] {$a_k'$};
      \draw (-2.4,0.4) -|(-2.4,0.6);
      \draw (-2.15,0.95) node[anchor=north west] {$b_{k-1}'$};
      \draw (-2,0.4) -|(-2,0.6);
      \draw (-0.55,0.95) node[anchor=north west] {$a_{k-1}'$};
      \draw (-0.45,0.4) -|(-0.45,0.6);
      \draw (-4.3,0.54)-- (-0.45,0.54);
      \draw [dash pattern=on 1pt off 1pt] (-3.3,0.46)-- (-2.4,0.46);
      \draw [dash pattern=on 1pt off 1pt] (-2,0.46)-- (-0.45,0.46);
      \draw (-3,0.4) node[anchor=north west] {$f^{s(k)-s(k-1)-1}(H_k)$};
      \draw (-3.25,1.2) node[anchor=north west] {$Q_{s(k)-1}$};
      \draw (-1.75,1.2) node[anchor=north west] {$Q_{s(k-1)-1}$};
      \draw (1.45,0.95) node[anchor=north west] {$a_{k-1}$};
      \draw (1.65,0.4) -|(1.65,0.6);
      \draw (2.15,0.95) node[anchor=north west] {$b_{k-1}$};
      \draw (2.45,0.4) -|(2.45,0.6);
      \draw (1.1,-0.65) node[anchor=north west] {$u$};
      \draw (1.25,-1.1) -|(1.25,-0.9);
      \draw (1.25,-1)-- (4,-1);
      \draw (3.85,-0.65) node[anchor=north west] {$c_1$};
      \draw (4,-1.1) -|(4,-0.9);
      \draw (2.80, -0.65) node[anchor=north west] {$y$};
      \draw (2.8,-1.0) -|(2.8,-0.9);
      \draw (1.45,-1.1) node[anchor=north west] {$a_{k-1}$};
      \draw (1.65,-1.1) -|(1.65,-0.9);
      \draw (2.15,-1.1) node[anchor=north west] {$b_{k-1}$};
      \draw (2.45,-1.1) -|(2.45,-0.9);
      \draw (1.6,-0.5) node[anchor=north west] {$Q_{s(k-1)}$};
      \draw (3.4,-1.1) node[anchor=north west] {$H_k$};
      \draw [dash pattern=on 1pt off 1pt] (1.65,-0.97) --(2.45,-0.97);
      \draw [dash pattern=on 1pt off 1pt] (1.65,0.46)-- (2.45,0.46);
      \draw (1.65,0.4) node[anchor=north west] {$Q_{s(k-1)}$};
      \draw (2.75,0.95) node[anchor=north west] {$a_{k}$};
      \draw (2.95,0.4) -|(2.95,0.6);
      \draw (3.25,0.95) node[anchor=north west] {$b_{k}$};
      \draw (3.35,0.4) -|(3.35,0.6);
      \draw [dash pattern=on 1pt off 1pt] (2.95,0.46)-- (3.35,0.46);
      \draw (2.9,0.4) node[anchor=north west] {$Q_{s(k)}$};
      \draw (1.1,0.95) node[anchor=north west] {$u$};
      \draw (1.25,0.4) -|(1.25,0.6);
      \draw (3.85,0.95) node[anchor=north west] {$c_1$};
      \draw (4,0.4) -|(4,0.6);
    \end{tikzpicture}
    \caption{On top, the double lined interval denote $f^{s(k)-s(k-1)-1}(H_k)$. On the bottom, we have the interval $H_k$.}
    \label{fig:B_k_H_k_graphic}
\end{figure}

To complete the proof, let $k$ be a large integer so that both $Q_{s(k)}$ and $Q_{s(k-1)}$ lies to the right hand side of $c$. So $a_{k-1}''=a_{k-1}'$. Let  
$$\alpha=\ell_+(\ell_+-1)/(\ell_++1).$$
We need to show that $$M_k:=\eta_{k-1}^\alpha |a_k-c_1|/|a_{k-1}-c_1|^{\ell_+}$$ is bounded away from zero. 
We again distinguish two cases. 

\noindent \textbf{Case 1.} $B_k^{s(k-1)}\not\ni c_1$. 

By the local behavior of $f$ near $c$, $|a''_{k-1}-a'_k|/|\tau(a''_{k-1})-a''_{k-1}|$ is bounded away from zero. By (\ref{eqn:ak-1ak}), 
$\eta_{k-1} |a'_k-c|/|\tau(a''_{k-1})-c|$ is bounded away from zero. Since  $|a_k'-c|^{\ell_+}$ is comparable to $|a_k-c_1|$, $|\tau(a''_{k-1})-c|$ is comparable to $|a_{k-1}-c_1|^{\ell_-}$ and $\ell_+\ge \ell_-\ge 1$, this implies that $M_k$ is bounded away from zero.

\textbf{Case 2.} $B_k^{s(k)-1}\ni c_1$. 

Let $x=f^{s(k)-s(k-1)}(c)$ and $g=f^{s(k)-s(k-1)-1}$ be as before, and let $\widetilde{x}=f^{s(k)-s(k-1)}(x)$. By (\ref{eqn:xc}), $|a_{k}'-b_k'|$ is much bigger than $|b_k'-x|$, the one-sided Koebe principle applies and give us 
$$|Dg(y)|\le C \frac{|a_k'-b_k'|}{|a_{k-1}-b_{k-1}|}\le C' \frac{|a_k'-c|}{|a_{k-1}-c_1|},$$
where $C, C'$ are constants and we used Lemma A for the second inequality.

Since $f^{s(k)-s(k-1)}: B_k\to B_k$ has a unique critical point $c$, we have $\widetilde{x}>c$, for otherwise, $f^{s(k)-s(k-1)}$ maps $[x,c]$ homeomorphically into itself, which implies that for any $z\in [x,c]$, $\omega(z)$ is a periodic orbit, contradicting with the fact that $[x,c]$ contains a wandering interval.
For similar reasons, $$f^{2(s(k)-s(k-1))}([c, \widetilde{x}])\supset [c, \widetilde{x}].$$  
Thus there exists $z\in [c,\widetilde{x}]$ such that 
$$|Df^{2(s(k)-s(k-1))}(z)|\ge 1.$$
Let $w=f^{s(k)-s(k-1)}(z)$. Then 
$$|Dg(f(z))|, |Dg(w)|\ge C'\frac{|a_k'-c|}{|a_{k+1}-c|}.$$
Since $z\in [c, b_k']$, we have 
$$|Dg(z)|\le A |b_k'-c|^{\ell^+-1}\le A' |b_{k}'-x|^{\ell_+-1} \le A'' \eta_{k-1}^{\ell_+-1} |a_k'-c|^{\ell_+-1}.$$
For $w$, we have 
$$|Dg(w)|\le 1.$$
Thus 
$$|Df^{2(s_k-s(k-1))}(z)|=|Df(z)||Dg(f(z))||Df(w)||Dg(w)|\le C'' \left(\frac{|a_k'-c|}{|a_{k-1}-c|}\right)^2 \eta_{k-1}^{\ell_+-1} |a_k'-c|^{\ell_+-1}.$$
Consequently, the right hand side is bounded away from zero. As $|a_k'-c|$ is comparable to $|a_k-c_1|^{1/\ell_+}$, it follows that 
$$\frac{|a_k-c_1|^{1/\ell_+}}{|a_{k-1}-c_1|^{2/(1+\ell_+)}}\eta^{\alpha}$$
is bounded away from zero, where $\alpha=(\ell_+-1)/(\ell_++1)>0$. Since $2/(\ell_++1)<1/\ell_-$.
$$\left(\frac{|a_k'-c|}{|a_{k-1}-c|}\right)^2 \eta_{k-1}^{\ell_+-1} |a_k'-c|^{\ell_+-1}$$
is bounded away from zero.  Since $|a_k'-c|^{\ell_+}$ is comparable to $|a_k-c_1|$ and $\ell_+>1$, this implies that $M_k$ is bounded away from zero. 

\end{proof}

\subsection{Proof of Lemma C}
\begin{proof} [Proof of Lemma C]
To be definite, let us assume $\ell_-\le \ell_+$. Then by Lemma~\ref{lem:position_of_Q_s(k)}, $Q_{s(k-1)}$ and $Q_{s(k)}$ both lie to the right of $c$, provided that $k$ is large enough. 
By Lemma~\ref{lem:Bk} (1), there is an interval $U\supset Q_{s(k-1)}$ such that $f^{s(k)-s(k-1)-1}$ maps $U$ diffeomorphically onto $(a_{k-1}',c)\subset B_k$. By Lemma~\ref{lem:Bk} (3) and Theorem~\ref{theo:CR_distortion_bound}, 
$$D((a'_{k-1},c), Q_{s(k)-1})\ge M D(U, Q_{s(k-1)}),$$
where $M>0$ is a constant. 
Let $y$ be the endpoint of $T_k^{s(k-1)}$ which is closer to $c_1$. 
That is, 
$$\frac{|b_{k-1}-y|}{|a_{k-1}-b_{k-1}|}\ge D(U, Q_{s(k-1)})^{-1}\ge M D((a'_{k-1},c), Q_{s(k)-1})^{-1}= M \frac{|b_k'-c|}{|a_k'-b_k'|}\frac{|a_k'-a_{k-1}'|}{|a_{k-1}'-c|}.$$
By Lemma~\ref{lem:etan0} and the local behavior of $f$ near $c$, for $k$ large enough, 
$$\frac{|b_{k-1}-y|}{|a_{k-1}-b_{k-1}|} \le 
\frac{|b_{k-1}-c_1|}{|a_{k-1}-c_1|}
\frac{|a_{k-1}-c_1|}{|a_{k-1}-b_{k-1}|}\le 2\eta_{k-1},$$
and
$$\frac{|b_k'-c|}{|a_k'-b_k'|}\frac{|a_k'-a_{k-1}'|}{|a_{k-1}'-c|}\ge \frac{|b_k'-c|}{|a_k'-c|}\frac{|b_{k-1}'-a_{k-1}'|}{|a_{k-1}'-c|}=\frac{1}{2}\eta_k^{1/\ell_+}.$$
The lemma follows.
\end{proof}

\bibliography{Biblio}{}
\bibliographystyle{alpha}
\end{document}